\documentclass[3p]{elsarticle}
\usepackage{amsmath,amsfonts,mathtools,mathdots,amssymb,amsthm}
\usepackage{thmtools,cleveref}
\usepackage{enumitem,multirow,array}
\usepackage{caption,subcaption}
\usepackage{algorithm,algorithmicx}
\usepackage{url}
\newcommand\scalemath[2]{\scalebox{#1}{\mbox{\ensuremath{\displaystyle #2}}}}

\DeclareMathOperator{\rank}{rank}
\DeclareMathOperator{\rev}{rev}

\DeclareMathOperator*{\argmin}{arg\,min}
\DeclareMathOperator{\Image}{Im}

\newcommand{\CC}{\mathbb{C}} 
\newcommand{\RR}{\mathbb{R}} 
 
\newcommand*{\Scale}[2][4]{\scalebox{#1}{\ensuremath{#2}}}
\newcommand{\bigzero}{\mbox{\normalfont\Large\bfseries 0}}
\newcommand{\bigI}{\mbox{\normalfont\Large\bfseries I}}

\newtheorem{thm}{Theorem}[section]
\newtheorem{prop}[thm]{Proposition}
\newtheorem{lem}[thm]{Lemma}
\newtheorem{cor}[thm]{Corollary}
\newtheorem{example}[thm]{Example}
\newtheorem{defn}[thm]{Definition}

\allowdisplaybreaks

\begin{document}

\begin{frontmatter}
\title{Real Factorization of Positive Semidefinite Matrix Polynomials \tnoteref{t1}}
\author{Sarah Gift\corref{cor1}}
	\ead{sg3664@drexel.edu}
\author{Hugo J. Woerdeman}
	\ead{hjw27@drexel.edu}
\address{Department of Mathematics, Drexel University, 
            Philadelphia, PA 19104, USA}
\tnotetext[t1]{This work was supported by the National Science Foundation Grant DMS 2000037 }
\cortext[cor1]{Corresponding author}

\begin{abstract}
Suppose $Q(x)$ is a real $n\times n$ regular symmetric positive semidefinite matrix polynomial.  Then it can be factored as
$$Q(x) = G(x)^TG(x),$$
where $G(x)$ is a real $n\times n$ matrix polynomial with degree half that of $Q(x)$ if and only if $\det(Q(x))$ is the square of a nonzero real polynomial.  We provide a  constructive proof of this fact, rooted in finding a skew-symmetric solution to a modified algebraic Riccati equation
$$XSX - XR + R^TX + P = 0,$$
where $P,R,S$ are real $n\times n$ matrices with $P$ and $S$ real symmetric.  In addition, we provide a detailed algorithm for computing the factorization.
\end{abstract}

\begin{keyword}
Positive semidefinite matrix polynomial \sep Algebraic Riccati equation \sep Matrix factorization
\MSC[2020] 47A68 \sep 46C20 \sep 15B48 \sep 93B05
\end{keyword}

\end{frontmatter}

\section{Introduction}
The Fej\'{e}r-Riesz factorization was first shown for matrix polynomials by Rosenblatt \cite{494} and Helson \cite{307}.  Its version on the real line is the following: given a matrix polynomial $Q(x) = \sum_{i=0}^{2m}Q_ix^i$ with $Q_i$ Hermitian and $Q(x)$ positive semidefinite for all $x\in\RR$, we can factorize it
as
$$Q(x) = G(x)^*G(x) , $$
where $G(x) =\sum_{i=0}^mG_ix^i$. In 1964, Gohberg generalized this factorization to certain operator-valued polynomials \cite{259}.  Later it was further generalized to operator-valued polynomials in general form \cite{495}.  The multivariable case has also been studied, e.g. in \cite{multivariable}.  For an overview of the work done with the operator-valued Fej\'{e}r-Riesz theorem, see \cite{500}.  Fej\'{e}r-Riesz factorization has applications in $H^\infty$-control
\cite{Francis}, in the construction of compactly supported wavelets \cite[Chapter 6]{Ten}, filter design \cite{Pextensions}, determinantal representations \cite{Stable}, and prediction theory \cite[Chapter XII]{Stochastic}, \cite{Ftrig, Zorzi}. In some cases, one may want to insist on having a real factorization. Indeed, our motivation for finding a real factorization came from an interest in constructing real symmetric solutions to A. Horn's problem, where the eigenvalues of two real symmetric matrices are prescribed, as well as the eigenvalues of their sum. Adjusting the techniques in \cite{Cao} to the real case required finding a real Fej\'{e}r-Riesz factorization. 

In this paper, we provide a constructive proof of the real analog of the Fej\'{e}r-Riesz factorization of matrix-valued polynomials.  In particular, given a matrix polynomial $Q(x) = \sum_{i=0}^{2m}Q_ix^i$ with $Q_i$ $n\times n$ real symmetric, $Q(x)$ positive semidefinite for all $x\in\RR$, and $\det(Q(x))$ equal to the square of a nonzero real polynomial, we show that $Q(x)$ admits the factorization
$$Q(x) = G(x)^TG(x) , $$
where $G(x)=\sum_{i=0}^mG_ix^i$ with $G_i$ real $n\times n$ matrices.
This result was first shown by Hanselka and Sinn \cite{Hanselka} using methods from projective algebraic geometry and number theory.  We provide an alternative, linear algebraic proof, inspired by the proof of the Fej\'{e}r-Riesz factorization presented in Section 2.7 of \cite{Moments}.  That proof in turn, was taken from \cite{187, Hachez}.   For earlier work on factorizations of real symmetric matrix polynomials (not necessarily positive semidefinite) see e.g. \cite{Symmetries}.

A key part of the Fej\'{e}r-Riesz factorization proof we follow requires finding a Hermitian solution to an algebraic Riccati equation
\begin{equation}\tag{\textasteriskcentered} \label{CARE}
XDX + XA + A^*X - C = 0,
\end{equation}
where $D$ and $C$ are Hermitian.  Reducing a factorization problem to solving a Riccati equation is a technique that has been used in many other papers as well  (see, e.g., \cite{Spectral}, \cite{Francis}, \cite[Chapter 19]{Ric} and references therein).  This technique is useful because Riccati equations have been studied extensively.  Early work was done by Willems \cite{Willems} and Coppel \cite{Coppel} in analyzing properties of solutions of continuous algebraic Riccati equations.  Another key paper was \cite{49} where Riccati equations were used to solve $H_\infty$-control problems. For an in depth analysis of algebraic Riccati equations, please see the book by Lancaster and Rodman \cite{Ric}.  

For the current real factorization problem, we end up needing to find a real skew-symmetric solution to an equation of the form
\begin{equation}\tag{\textasteriskcentered\textasteriskcentered}\label{MCARE}
XSX - XR + R^TX + P = 0,
\end{equation}
where  $P$ and $S$ are real symmetric.  This is not quite an algebraic Riccati equation of the form (\ref{CARE}) and thus we call it a modified algebraic Riccati equation.  In general, to find a skew-Hermitian solution $X$, one often considers instead $iX$, which is Hermitian; however, we want real solutions and thus this method is not applicable here.  Thus we instead follow the same steps presented in \cite{Ric} for finding a real \textit{symmetric} solution to the real algebraic Riccati equation (\ref{CARE}) and amend them to our current situation.  This is the topic of \Cref{ricatti}, which culminates in giving sufficient conditions for a skew-symmetric solution of our modified Riccati equation (\ref{MCARE}).  In \Cref{poly} we provide additional background on matrix polynomials necessary for our main result, the real factorization of a symmetric positive semidefinite matrix polynomial, presented in \Cref{main}.  The major advantage of our proof compared to that by Hanselka and Sinn \cite{Hanselka} is that ours is constructive.  Thus we provide an explicit algorithm for finding the factorization, along with examples illustrating the construction.  
 
\section{A Modified Algebraic Riccati Equation}\label{ricatti}

The goal of this section is to provide necessary and sufficient conditions for the existence of a real skew-symmetric solution $X$ to the modified algebraic Riccati equation
\begin{equation}
XSX - XR + R^TX + P = 0 \label{RIC},
\end{equation}
where $P,R,S$ are real $n\times n$ matrices with $P$ and $S$ real symmetric.  
In Chapter 8 of the book \textit{Algebraic Riccati Equations} \cite{Ric}, Lancaster and Rodman show conditions for which there is a real symmetric solution $X$ to the continuous time algebraic Riccati equation
$$XDX + XA + A^TX - C = 0,$$ 
where $A,C,D$ are real $n\times n$ matrices with $C$ and $D$ real symmetric.  We amend these results to the present situation.
Define the $2n\times 2n$ real matrices
\begin{equation}\label{Mr}
M_r = \begin{bmatrix} R & -S \\  P & R^T\end{bmatrix}, \qquad
\hat{H}_r = \begin{bmatrix} 0 & I \\ I & 0 \end{bmatrix}, \qquad
H_r = \begin{bmatrix} P & R^T \\ R & -S \end{bmatrix}.
\end{equation}
Then both $\hat{H}_r$ and $H_r$ are real symmetric.  Also 
$$\hat{H}_rM_r = M_r^T\hat{H}_r \qquad \text{and} \qquad H_rM_r = M_r^TH_r.$$
Using terminology from \cite[Section 2.6]{Ric}, we say $M_r$ is both $H_r$-symmetric and $\hat{H}_r$-symmetric. Next, we define the graph subspace of a real $n\times n$ matrix $X$ by
$$G(X) := \Image \begin{bmatrix} I_n \\ X \end{bmatrix} = \left\{ \begin{bmatrix} I_n \\ X \end{bmatrix} x : x\in \RR^{n}\right\}.$$
We can now give a condition for a real solution of \cref{RIC} to exist.

\begin{prop}
\label{P711}
$X$ is a real solution of \cref{RIC} if and only if the graph subspace $G(X)$
is $M_r$-invariant, where $M_r$ is defined as in \cref{Mr}.
\end{prop}
\begin{proof}
If $G(X)$ is $M_r$-invariant, then 
\begin{equation}\label{invariant}
	\begin{bmatrix} R & -S \\ P & R^T\end{bmatrix} 
	\begin{bmatrix} I \\ X \end{bmatrix} 
	=  \begin{bmatrix} I \\ X \end{bmatrix} Z
\end{equation}
for some $n\times n$ matrix $Z$.  The first block row gives $Z = R-SX$ and the second gives
$P + R^TX = XZ$.  Combining the two gives
$$ P + R^TX = X(R - SX).$$
Thus $X$ solves \cref{RIC}.  Conversely, if $X$ solves \cref{RIC}, then \cref{invariant} holds for $Z = R-SX$ and thus $G(X)$ is $M_r$-invariant.
\end{proof}

\noindent
More than just a real solution, though, we want a skew-symmetric solution.  Thus we next strive to give a condition for such a solution.  For this we first need a few definitions (see \cite[Section 2.6]{Ric}).  For $x,y\in\RR^n$, we use the notation $\langle x, y \rangle$ to mean the usual inner product $y^Tx$.
\begin{defn}
\rm
Let $H$ be an $n\times n$ invertible real symmetric matrix.  A subspace $\mathcal{M}$ of $\RR^n$ is called 
\begin{enumerate}
	\item \textit{$H$-nonnegative} if $\langle Hx, x \rangle \ge 0$ for all $x\in \mathcal{M}$.
	\item \textit{$H$-nonpositive} if $\langle Hx, x \rangle \le 0$ for all $x\in \mathcal{M}$.
	\item \textit{$H$-neutral} if $\langle Hx, x \rangle = 0$ for all $x\in \mathcal{M}$.
\end{enumerate}
\end{defn}

\begin{prop}
\label{P811}
\text{ }
\begin{enumerate}
	\item For $\hat{H}_r$ as defined in \cref{Mr}, $X\in{\mathbb R}^{n\times n}$ is skew-symmetric if and only if the graph subspace $G(X)$ is $\hat{H}_r$-neutral.
	\item Let $X$ be a real solution of \cref{RIC}.  Then for $H_r$ as defined in \cref{Mr}, $G(X)$ is $H_r$-nonpositive if and only if $(X^T+X)(R-SX)$ is negative semi-definite.
\end{enumerate}
\end{prop}

\begin{proof}
\begin{enumerate}
	\item $G(X)$ is $\hat{H}_r$-neutral if and only if for all $z\in\RR^n$,
	$$
			\left\langle \hat{H}_r \begin{bmatrix} I \\ X\end{bmatrix} z, \begin{bmatrix} I \\ X\end{bmatrix} z \right \rangle  = 0 .$$
		Rewriting the lefthand side, we have
			$$	\left\langle \hat{H}_r \begin{bmatrix} I \\ X\end{bmatrix} z, 
					\begin{bmatrix} I \\ X\end{bmatrix} z \right \rangle
				 = z^T \begin{bmatrix} I & X^T\end{bmatrix} 
					\hat{H}_r
					 \begin{bmatrix} I \\ X\end{bmatrix} z 
				 = z^T (X + X^T) z.$$
Thus $G(X)$ is $\hat{H}_r$-neutral if and only if $X + X^T = 0$, i.e. $X$ is skew-symmetric.
	\item  $G(X)$ is  $H_r$-nonpositive if and only if for all $z\in \RR^{n}$,
		$$
			\left\langle H_r \begin{bmatrix} I \\ X\end{bmatrix} z, \begin{bmatrix} I \\ X\end{bmatrix} z \right \rangle  \le 0 .$$
		Rewriting the lefthand side, we have
		\begin{align*}
		\left\langle H_r \begin{bmatrix} I \\ X\end{bmatrix} z, \begin{bmatrix} I \\ X\end{bmatrix} z \right \rangle
				& =  z^T \begin{bmatrix} I & X^T\end{bmatrix} 
					\begin{bmatrix} P & R^T \\ R & -S \end{bmatrix}
					 \begin{bmatrix} I \\ X\end{bmatrix} z  \\
				& =  z^T [P+R^TX + X^TR - X^TSX] z   \\
				 & = z^T [XR - XSX + X^TR - X^TSX] z   \qquad\text{by \cref{RIC} } \\
				& =  z^T (X^T+X)(R-SX) z .
		\end{align*}
		Thus $G(X)$ is  $H_r$-nonpositive if and only if $(X^T+X)(R-SX)$ is negative semidefinite.
\end{enumerate}
\end{proof}

\noindent
\Cref{P811} shows that in order to get a real skew-symmetric solution $X$ to the \cref{RIC}, we need an $\hat{H}_r$-neutral subspace $G(X)$ of dimension $n$.  We consider now conditions for such a subspace to exist.  For this we first state a few known results (see \cite[Section 2.6]{Ric}).

\begin{defn}
\rm
Let $A$ be a square matrix and $\lambda_i$ be an eigenvalue of $A$.  We call the sizes of the Jordan blocks of $\lambda_i$ the \textit{partial multiplicities} of $\lambda_i$.
\end{defn}

\begin{thm}\cite[Part of Theorem 2.6.3]{Ric} \label{A_invariant} 
Let $A$ be a real $n\times n$ $H$-symmetric matrix (meaning $H$ is symmetric and $HA = A^TH$). Suppose the partial multiplicities of the real eigenvalues of $A$ are all even.  Then there exists an $A$-invariant $H$-neutral subspace of dimension $k-p$ where $k$ is the number of positive eigenvalues of $H$ (counting algebraic multiplicities) and $p$ is the number of distinct pairs of non-real complex conjugate eigenvalues of $A$ with odd algebraic multiplicity.
\end{thm}

\begin{lem}\cite[p. 56]{Ric}
\label{L56}
Let $H$ be a real symmetric matrix.  A real subspace $\mathcal{M}$ is $H$-neutral if and only if $\langle Hx, y \rangle = 0$ for all $x,y\in \mathcal{M}$.
\end{lem}
\begin{proof}
This follows from the relation
\begin{align*}
\langle Hx, y \rangle = \frac{1}{2}\left( \langle H(x+y),x+y \rangle - \langle Hx, x \rangle - \langle Hy, y \rangle \right).
\end{align*}
\end{proof}

\noindent 
Now we are ready for the following new result.

\begin{lem}
\label{L812}
	Let \cref{Mr} hold with $P$ and $S$ real symmetric.  Among the following statements, the implications $(iii)\implies(ii)\implies(i)$ hold.
	\begin{enumerate}[label=(\roman*)]
		\item There exists an $n$-dimensional $M_r$-invariant $H_r$-neutral subspace.
		\item There exists an $n$-dimensional $M_r$-invariant $\hat{H}_r$-neutral subspace.
		\item All real eigenvalues of $M_r$ have even partial multiplicities and all non-real eigenvalues of $M_r$ have even algebraic multiplicity.
	\end{enumerate}
	If in addition $M_r$ is invertible, then we also have $(i)\implies(ii)$.
\end{lem}
\begin{proof}
	Assume (iii).  By \Cref{A_invariant}, there is an $n$-dimensional $M_r$-invariant $\hat{H}_r$-neutral subspace (since $p = 0$ and $\hat{H}_r$ has $k = n$ positive eigenvalues).  Thus $(iii)\implies (ii)$.  Next, since $H_r = \hat{H_r}M_r$, it follows by \Cref{L56} that $(ii) \implies (i)$.  Finally, if $M_r$ is invertible, $\hat{H_r} = H_rM_r^{-1}$ and thus again by \Cref{L56}, $(i) \implies (ii)$.
\end{proof}

\noindent We need one more result before the main theorem of this section.  For this result, we first recall a definition (see, e.g., \cite[Section 4.1]{Ric}).
\begin{defn}
\rm
Let $A$ be a real $n\times n$ matrix and $B$ be a real $n\times m$ matrix. The pair $(A,B)$ is said to be \textit{controllable} if
$$\rank \begin{bmatrix} B & AB & A^2B & \cdots & A^{n-1}B\end{bmatrix} = n.$$
\end{defn}

\noindent 
For an $n\times n$ real symmetric matrix $S$, we use the notation $S>0$ to mean $S$ is positive definite and the notation $S\ge0$ to mean $S$ is positive semidefinite.
\begin{lem}
 \label{L722}
Let \cref{Mr} hold with $P$ and $S$ real symmetric.  Assume that $S\ge 0$ and the pair $(R,S)$ is controllable.  Let $\mathcal{L}$ be an $n$-dimensional $M_r$-invariant $H_r$-nonnegative subspace of $\RR^{2n}$.  Then $\mathcal{L}$ is a graph subspace, i.e.
$$\mathcal{L} = \Image \begin{bmatrix} I \\ X \end{bmatrix}$$
for some real $n\times n$ matrix $X$.
\end{lem}
\begin{proof}
For $\mathcal{L}$ as defined in the statement, write
$$\mathcal{L} = \Image \begin{bmatrix} X_1 \\ X_2 \end{bmatrix}$$
for some real $n\times n$ matrices $X_1$ and $X_2$.  We shall show that $X_1$ is invertible.  First, since $\mathcal{L}$ is $M_r$-invariant,
$$\begin{bmatrix} R & -S \\ P & R^T\end{bmatrix}
	 \begin{bmatrix} X_1 \\ X_2 \end{bmatrix} =  \begin{bmatrix} X_1 \\ X_2 \end{bmatrix}T$$
for some $n\times n$ matrix $T$.  Thus,
\begin{align}
	RX_1 - SX_2 & = X_1T , \label{RS} \\
	PX_1 + R^T X_2 & = X_2T . \label{PRT}
\end{align}
Next, since $\mathcal{L}$ is $H_r$-nonnegative, we know
\begin{equation}\label{PSD}
 \begin{bmatrix} X_1^T & X_2^T \end{bmatrix}\begin{bmatrix} P & R^T \\ R & -S \end{bmatrix}\begin{bmatrix} X_1 \\ X_2 \end{bmatrix} = X_1^TPX_1+X_1^TR^TX_2+X_2^TRX_1-X_2^TSX_2
 \ge 0 .\end{equation}
Let $\mathcal{K} = \ker X_1$.  By \cref{PSD}, for every $x\in\mathcal{K}$,
$$0\le x^TX_1^TPX_1x+x^TX_1^TR^TX_2x+x^TX_2^TRX_1x-x^TX_2^TSX_2x = -x^TX_2^TSX_2x.$$
Since $S\ge0$, $X_2x\in\ker S$, so 
$$X_2\mathcal{K}\subset\ker S.$$
Then, \cref{RS} implies
$$T\mathcal{K}\subset\mathcal{K}.$$
Consequently, \cref{PRT} gives
$$R^TX_2\mathcal{K}\subset X_2\mathcal{K}.$$
All together, we have
$$R^TX_2\mathcal{K}\subset \ker S.$$
By induction, we get
$$(R^T)^rX_2\mathcal{K}\subset \ker S, \qquad r = 0,1,2,\dots$$
Now for every $x\in \mathcal{K}$,
$$\begin{bmatrix} S \\ SR^T \\ \vdots \\ S(R^T)^{n-1}\end{bmatrix} (X_2x) = 0.$$
Since $(R,S)$ is controllable, we must have
$X_2x = 0$.  The only $n$-dimensional vector $x$ for which $X_1x=X_2x=0$ is the zero vector (otherwise $\dim\mathcal{L}<n$).  Thus $\mathcal{K} = \{0\}$ and $X_1$ is invertible.  Hence
$$\mathcal{L} =\Image \begin{bmatrix}  I\\ X \end{bmatrix},$$
where $X = X_2X_1^{-1}$.  Thus $\mathcal{L}$ is a graph subspace.
\end{proof}

\noindent Now we put everything together to get necessary and sufficient conditions for the existence of a real skew-symmetric solution of \cref{RIC}.

\begin{thm}\label{T8110}
Let \cref{Mr} hold with $P$ and $S$ real symmetric.  Assume that $S\ge0$ and the pair $(R,S)$ is controllable.  Then the following are equivalent.
\begin{enumerate}[label=(\roman*)]
	\item \Cref{RIC} has a real skew-symmetric solution.
	\item There exists an $n$-dimensional $M_r$-invariant $\hat{H_r}$-neutral subspace.
\end{enumerate}
\end{thm}
\begin{proof}
By \Cref{P711} and \Cref{P811}, we know that if $X$ is a real skew-symmetric solution to \cref{RIC}, then $G(X)$ is an $n$-dimensional $M_r$-invariant $\hat{H}_r$-neutral subspace.  Thus $(i)\implies(ii)$.  Finally, assume there exists an $n$-dimensional $M_r$-invariant $\hat{H}_r$-neutral subspace, say $\mathcal{L}$.  Since $H_r = \hat{H}_rM_r$, $\mathcal{L}$  is also an $n$-dimensional $M_r$-invariant $H_r$-neutral subspace.  Clearly $\mathcal{L}$ is an $H_r$-nonnegative subspace, so by \Cref{L722}, $\mathcal{L}$ is a graph subspace.  Since $\mathcal{L} = G(X) $ is $M_r$-invariant, $X$ is a real solution of \cref{RIC} by \Cref{P711}.  Since $\mathcal{L} = G(X) $ is also $\hat{H}_r$-neutral, by \Cref{P811}, $X$ is skew-symmetric. Thus $(ii)\implies(i)$.
\end{proof}

\section{Matrix Polynomials}\label{poly}

Building toward our goal of factorizing a real symmetric positive semidefinite matrix polynomial, we next state a few relevant results on matrix polynomials (see \cite{Polynomials}).  
\begin{defn}
\rm
	For $n\times n$ matrices $P_i$, we define an \textit{$n\times n$ matrix polynomial $P(x)$ of degree $m$} by
	$$P(x) = \sum_{i=0}^m P_i x^i.$$
	\begin{itemize}
		\item The matrix polynomial is called \textit{real} if all $P_i$ are real matrices.
		\item The matrix polynomial is called \textit{monic} if $P_m = I_n$.  
		\item The matrix polynomial is called \textit{self-adjoint} if $P_i = P_i^*$, its conjugate transpose, for all $i$.
		\item The matrix polynomial is called \textit{symmetric} if $P_i = P_i^T$ for all $i$.
		\item The matrix polynomial is called \textit{positive semidefinite (also nonnegative)} if for all $x\in\RR$, $P(x)$ is positive semidefinite.
		\item The matrix polynomial is called \textit{regular} if $\det(P(x))$ is not identically zero.
	\end{itemize}
\end{defn}

Following \cite{Review}, where the spectrum and Jordan canonical form of a quadratic matrix polynomial were defined, we define these concepts for any degree matrix polynomial in the following ways.\footnote{Note that the partial multiplicities of an eigenvalue of a matrix polynomial are often defined as powers of the elementary divisors; however, these partial multiplicities are the same as the sizes of the Jordan blocks of our companion matrix. See the Appendix of \cite{Polynomials} for a more in depth understanding of matrix polynomial equivalences, linearizations, partial multiplicities and elementary divisors.}

\begin{defn}
\rm
Let $P(x)$ be a regular matrix polynomial.  Then the set of \textit{eigenvalues} of $P$, i.e. the \textit{spectrum}, is
$$\sigma(P):=\{x\in\CC: \det(P(x)) = 0\}.$$
\end{defn}
\begin{defn}
\rm
Let $P(x) = \sum_{i=0}^{m}P_ix^i$ be a degree $m$ matrix polynomial with $P_m$ nonsingular.
  The \textit{Jordan canonical form} for $P(x)$ is defined to be that of the companion matrix
$$C_P:=\begin{bmatrix}
			0 & I_n & 0 & \cdots & 0 \\
			0 & 0 & I_n & \cdots & 0 \\
			\vdots & \vdots & \ddots & \ddots & \vdots \\
			0 & 0 & \cdots & 0 & I_n \\
			-P_m^{-1}P_0 & -P_m^{-1}P_1 & \cdots & -P_m^{-1}P_{m-2} & -P_m^{-1}P_{m-1} 
			\end{bmatrix}.$$
\end{defn}

\begin{thm}\cite[part of Theorem 12.8]{Polynomials} \label{multiplicities} 
For a monic self-adjoint matrix polynomial $P(x)$, the following statements are equivalent.
\begin{enumerate}[label=(\roman*)]
	\item $P(x)$ is nonnegative.
	\item The partial multiplicities of $P(x)$ for real points of the spectrum are all even.
\end{enumerate}
\end{thm}

In the previous section, we found that if all real eigenvalues of $M_r$ have even partial multiplicities and all non-real eigenvalues of $M_r$ have even algebraic multiplicity, then our \cref{RIC} has the desired real skew-symmetric solution.  
Next, using the notion of linearization, we will associate the matrix $M_r$ with a monic non-negative matrix polynomial.  For this we begin with a few definitions from \cite[Section 1.1]{Polynomials}.

\begin{defn}
\rm
Two matrix polynomials $M_1(x)$ and $M_2(x)$ of size $n \times n$ are called \textit{equivalent} (notated $M_1(x)\sim M_2(x)$) if 
$$M_1(x) = E(x)M_2(x)F(x)$$
for some $n \times n$ matrix polynomials $E(x)$ and $F(x)$ with constant nonzero determinants.
\end{defn}

\begin{defn}
\rm
Let $P(x)$ be an $n\times n$ monic matrix polynomial of degree $m$.  A linear matrix polynomial $x I_{nm} - A$ is called a \textit{linearization} of $P(x)$ if 
$$
x I_{nm} - A \sim \begin{bmatrix} P(x) & 0 \\ 0 & I_{n(m-1)} \end{bmatrix}.
$$
\end{defn}

\noindent
Note that $x I - C_P$ is a linearization of $P(x)$.  For any linearization $x I - A$, the partial multiplicities in every eigenvalue of $A$ and $P(x)$ are the same \cite[Section 1.1]{Polynomials}.  

\begin{lem}\label{Mr_lem}
Let $Q(x)= \sum_{j=0}^{2m}Q_jx^j$ be an $n\times n$ real symmetric matrix polynomial of degree $2m$ with $Q_0=I_n$.
If $m=1$, set
$$M_r := \begin{bmatrix}
			-\frac{1}{2}Q_1 & -I_n \\[5pt]
			Q_2-\frac{1}{4}Q_1^2 & -\frac{1}{2}Q_1
		\end{bmatrix}.$$
Otherwise, define the $2nm\times 2nm$ matrix as
$$M_r :=\scalemath{0.8}{
	\left[\begin{array}{ccccc|ccccc}
		-\frac{1}{2}Q_1 &&&&&		-I_n &&& \\[5pt]
		I_n & 0 &&&&			 &0  &&& \\
		&I_n & 0 &&&			& & 0 && \\
		&& \ddots & \ddots &	&	&& & \ddots & \\
		&&& I_n & 0	&		&&& & 0 \\[5pt]\hline
		&&&&&&&&&\\
		Q_2-\frac{1}{4}Q_1^2 & \frac{1}{2}Q_3 & & &&		-\frac{1}{2}Q_1 &I_n&& \\[5pt]
		\frac{1}{2}Q_3 & Q_4 & \frac{1}{2}Q_5 & &&				 & 0 &I_n& \\[5pt]
		&\frac{1}{2}Q_5 & Q_6 & \ddots 	&		&&  & 0 & \ddots& \\[5pt]
		&&\ddots&\ddots & 		\frac{1}{2}Q_{2m-1} & &&&\ddots&I_n \\[5pt]
		&&& \frac{1}{2}Q_{2m-1} & Q_{2m}&				&&& & 0
	\end{array}\right]}.$$
Then for
$$\rev Q(x) := \sum_{i=0}^{2m} Q_{2m-i} x^{i} = x^{2m}I_n +\sum_{i=0}^{2m-1} Q_{2m-i} x^{i} ,$$
we have
$$ x I_{2nm} - M_r \sim \begin{bmatrix} \rev Q(x) & 0 \\ 0 & I_{2nm-n} \end{bmatrix}.$$
\end{lem}

\begin{proof}
If $m=1$,
$$\begin{bmatrix}
	 I_n & -xI_n-\frac{1}{2}Q_1 \\[5pt]
	0 & I_n  
	\end{bmatrix}
	\begin{bmatrix} 
		0 & I_n \\
		I_n & 0
	\end{bmatrix} (xI_{2n}-M_r)
	= 
	\begin{bmatrix}
		-Q_2-xQ_1-x^2I_n & 0 \\[5pt]
		xI_n+\frac{1}{2}Q_1 & I_n
		\end{bmatrix}.$$
It is clear then
$$ x I_{2n} - M_r \sim \begin{bmatrix} \rev Q(x) & 0 \\ 0 & I_{n} \end{bmatrix}.$$
Now assume $m\ge2$.  We begin by defining $M_1(x)$ as a permutation of the rows and columns  of $x I - M_r$.
\begin{align*}
M_1(x)&\scalemath{0.8}{:= \left[\begin{array}{cccc|ccccc}
		 &&&0&				0 &\cdots&0& I_n \\[5pt]
		 & &0&I_n&			 \vdots &  &0&0 \\
		& \iddots & \iddots &	&	& \iddots&  &\vdots \\
		0&I_n&& & 			0&&\cdots & 0 \\[5pt]\hline
		&&&&&&&&\\
		I_n & 0 & \cdots&0 &					0 &&& \\[5pt]
		0 & 0 &  & \vdots &				I_n & 0 & & \\[5pt]
		\vdots& & \ddots &0	&		&  \ddots& \ddots & \ \\[5pt]
		0&\cdots&  & 0&				&& I_n& 0
	\end{array}\right]
	(x I - M_r )
	\left[\begin{array}{c|c}
			\begin{matrix}
				 & & I_n \\
				& \iddots & \\
				I_n & & 
			\end{matrix}
			& \bigzero \\[5pt] \hline
			&\\
			\bigzero &
			\begin{matrix}
				I_n & & \\
				& \ddots & \\
				& & I_n
			\end{matrix}
	\end{array}\right]} \\[10pt]
	& = 
	\scalemath{0.8}{\left[\begin{array}{ccccc|ccccc}
		-Q_{2m}&-\frac{1}{2}Q_{2m-1}&0&  & &				
											&& & x I_n \\[5pt]
		x I_n & - I_n &&&&			 
											&  &0& \\
		& \ddots & \ddots &	&&	
											&  \iddots&& \\
		&& \ddots & \ddots &	&	
											& & & \\
		&&& x I_n & - I_n	&		
											0&& &  \\[5pt]\hline
		&&&&&&&&&\\
		 &&&&x I_n+\frac{1}{2}Q_1&		
											I_n &&& \\[5pt]
		 &  &&-\frac{1}{2}Q_3&\frac{1}{4}Q_1^2-Q_2&		
											x I_n+\frac{1}{2}Q_1 & -I_n&& \\[5pt]
		 & & \iddots &-Q_4&-\frac{1}{2}Q_3&			 
											& x I_n &\ddots& \\[5pt]
		&\iddots & \iddots & \iddots&	&	
											&  & \ddots & \ddots \\[5pt]
		-\frac{1}{2}Q_{2m-1}& -Q_{2m-2}&-\frac{1}{2}Q_{2m-3} & & 	&		
											&  &  & x I_n & -I_n \\[5pt]
	\end{array}\right].}
\end{align*}
Clearly, 
$$x I - M_r \sim M_1(x).$$
Next, define $W,V_1,\dots,V_{m-2}$ by
\begin{align*}
		 W & = -\frac{1}{2}x^{m-2}Q_3 - x^{m-1}Q_2 - x^mQ_1 -x^{m+1}I_n \\[5pt]
		 V_1 & = -\frac{1}{2}x^{m-3}Q_5-x^{m-2}Q_4-\frac{1}{2}x^{m-1}Q_3 \\[5pt]
		  V_2 & = -\frac{1}{2}x^{m-4}Q_7-x^{m-3}Q_6-\frac{1}{2}x^{m-2}Q_5 \\[5pt]
		 	&\vdots \\
		V_{m-3} & = -\frac{1}{2}x Q_{2m-3}-x^2Q_{2m-4} - \frac{1}{2}x^3Q_{2m-5} \\[5pt]
		V_{m-2} &= -\frac{1}{2}Q_{2m-1} - x Q_{2m-2} - \frac{1}{2}x^2 Q_{2m-3}.
\end{align*}
Then set
$$V(x) : = \scalemath{0.68}{\left[\begin{array}{c|c|c}
		I_n & \begin{matrix}
				x^{m-2}W+ \sum_{i=1}^{m-2} x^{m-2-i}V_{i}
				& x^{m-3}W+\sum_{i=1}^{m-3} x^{m-3-i}V_{i}  
				&\cdots
				&xW + V_{1} 
				& W
				\end{matrix}
			& \begin{matrix}
	 		-x^{m-1}(x I_n+\frac{1}{2}Q_1)
				&x^{m-1}I_n
				&\cdots
				&x^2 I_n
				&x I_n \\
			\end{matrix} \\[5pt]
			\bigzero & \bigI_{n(m-1)} & \bigzero \\[5pt]\hline
			&& \\
		\bigzero &\bigzero& \bigI_{nm}
		\end{array}\right]}.$$
Noting that
$$-Q_{2m} - \frac{1}{2} x Q_{2m-1} + x\left(x^{m-2}W+ \sum_{i=1}^{m-2} x^{m-2-i}V_{i}\right)= -\rev Q(x),$$
we get 
$$ V(x)M_1(x) = 
	\left[  		
	\begin{array}{c|c}
		\begin{matrix}
		-\rev Q(x) &&&\\
		 & - I_n &&\bigzero&	\\
		{\Scale[2]{*}}&&  & \ddots &		\\
		&&&   & - I_n	
		\end{matrix} 
		&\bigzero \\ \hline
		{\Scale[4]{*}} & 
		\begin{matrix}	
											I_n &&& \\
											& -I_n&&\bigzero \\	 	
											& {\Scale[2]{*}} &  & \ddots \\		
											&  &  & & -I_n 
											\end{matrix}
		
	\end{array}\right].$$
Since $V(x)$ has constant nonzero determinant, $x I - M_r \sim V(x)M_1(x)$. It is evident now that 
\begin{align*}
 x I_{2nm} - M_r \sim \begin{bmatrix} \rev Q(x) & 0 \\ 0 & I_{2nm-n} \end{bmatrix}.
\end{align*}
\end{proof}

\section{Real Factorization of Non-negative Matrix Polynomial}\label{main}

We are now ready for the main result.  While the following theorem was previously proven by Hanselka and Sinn \cite{Hanselka}, we provide a new constructive proof following that of the complex analogue presented in the monograph by Bakonyi and Woerdeman \cite[Section 2.7]{Moments}.
\begin{thm}\label{sq_factor}
Let $Q(x)= \sum_{j=0}^{2m}Q_jx^j$ be an $n\times n$ real symmetric positive semidefinte matrix polynomial of degree $2m$ with $Q_0>0$.  Then the roots of $\det(Q(x))$ all have even multiplicity if and only if there exists an $n\times n$ real matrix polynomial $G(x)= \sum_{j=0}^{m}G_jx^j$
of degree $m$ such that
$$Q(x) = G(x)^TG(x).$$
\end{thm}
\begin{proof}
First assume $Q(x) = G(x)^TG(x).$ Then $\det(Q(x)) = \det(G(x))^2$, so clearly all roots have even multiplicity.  On the other hand, assume the roots of $\det(Q(x))$ all have even multiplicity.
Without loss of generality, assume $Q_0=I_n$ (otherwise, take $\tilde{Q}(x):=Q_0^{-1/2}Q(x)Q_0^{-1/2}$).
Consider the $(m+1)n\times (m+1)n$ real symmetric matrix
$$F_0 = \begin{bmatrix}
			I_n & \frac{1}{2}Q_1 & & \\[10pt]
			\frac{1}{2}Q_1 & Q_2 & \ddots & \\[10pt]
			&\ddots & \ddots & \frac{1}{2}Q_{2m-1} \\[10pt]
			&& \frac{1}{2}Q_{2m-1} & Q_{2m}
		\end{bmatrix}.$$
Given an $nm\times nm$ real skew-symmetric matrix $X$,
let 
$$F_X = F_0 + \begin{bmatrix} 0_{nm\times n} & X \\ 0_n & 0_{n\times nm} \end{bmatrix} - \begin{bmatrix} 0_{n\times nm} & 0_n \\ X & 0_{nm\times n}\end{bmatrix}.$$
It should be noted that in the above line, the matrix decompositions are different; e.g. the $X$ block and the $-X$ block overlap in general.  We want to solve
\begin{align*}
	X_{\text{opt}} &= \argmin \ \ \rank \left(F_X\right) 
		\quad \text{ such that } F_X\ge0.
\end{align*}
Let 
$$A = \begin{bmatrix}
		0_n &&& \\
		I_n & 0_n && \\
		& \ddots & \ddots & \\
		&& I_n & 0_n
	\end{bmatrix} \in\RR^{nm\times nm}
\qquad\text{and}\qquad
B = \begin{bmatrix}
	I_n \\
	0_n\\
	\vdots\\
	0_n
	\end{bmatrix} \in \RR^{nm\times n}.$$
Then $(A,B)$ is controllable, 
$$\begin{bmatrix}
	0_{n\times nm} & 0_n \\
	X & 0_{nm\times n}
	\end{bmatrix}
	=\begin{bmatrix}
		0_{n\times n} &0_{n\times nm} \\
		XB & XA
	\end{bmatrix},
\qquad\text{and}\qquad
\begin{bmatrix}
	0_{nm\times n} & X \\
	0_n & 0_{n\times nm}
	\end{bmatrix}
	=\begin{bmatrix}
		0_{n\times n} &B^TX \\
		0_{nm\times n} & A^TX
	\end{bmatrix}.$$
Split $F_0$ into four blocks as follows,
$$F_0 = 
\left[\begin{array}{c|c}
			I_n & \begin{matrix} \frac{1}{2}Q_1 
								&\phantom{ \frac{1}{2}Q_{2m-1}} 
								& \phantom{ \frac{1}{2}Q_{2m-1}} 
				\end{matrix} \\[5pt] \hline
			\begin{matrix}
			\frac{1}{2}Q_1 \\[10pt]
			 \phantom{\frac{1}{2}}\\[10pt]
			 \phantom{\frac{1}{2}}
			\end{matrix}
			& \begin{matrix}
			 Q_2 & \ddots & \\[10pt]
			\ddots & \ddots & \frac{1}{2}Q_{2m-1} \\[10pt]
			& \frac{1}{2}Q_{2m-1} & Q_{2m}
		\end{matrix} \end{array}\right]
		=: \begin{bmatrix}
			I_n & \Gamma_{12} \\
			\Gamma_{21} & \Gamma_{22}
		\end{bmatrix}.$$
Noting $\Gamma_{12}^T = \Gamma_{21}$ and $\Gamma_{22}$ is real symmetric, we can recast the condition
$F_X\ge0$
as
$$\begin{bmatrix}
	I_n & \Gamma_{12}+B^TX \\
	\Gamma_{12}^T-XB & \Gamma_{22}+A^TX - XA
	\end{bmatrix} \ge0 .$$
Consider the Schur complement with respect to $I_n$,
$$\Gamma_{22}+A^TX - XA - (\Gamma_{12}^T-XB)I_n^{-1}(\Gamma_{12}+B^TX).$$
Setting this equal to zero, we get the modified algebraic Riccati equation
\begin{equation}\label{QARE}
P + R^TX - XR + XSX = 0,
\end{equation}
where
\begin{align*}
	P &= \Gamma_{22} - \Gamma_{12}^T\Gamma_{12} ,\\
	R &= A - B\Gamma_{12} ,\\
	S &= BB^T.
\end{align*}
Note $P = P^T$ and $S = S^T$ with $S\ge0$.
In this case, the associated $M_r$ matrix is 
\begin{align*}
M_r &= \begin{bmatrix}
			A - B\Gamma_{12} & -BB^T \\
			\Gamma_{22}-\Gamma_{12}^T\Gamma_{12} & A^T - \Gamma_{12}^TB^T
			\end{bmatrix} \\[5pt]
	&=
	\left[\begin{array}{cccc|ccccc}
		-\frac{1}{2}Q_1 &&&&		-I_n &&& \\[5pt]
		I_n & 0 &&&			 & 0 && \\
		& \ddots & \ddots &	&	& & \ddots & \\
		&& I_n & 0	&		&& & 0 \\[5pt]\hline
		&&&&&&&&\\
		Q_2-\frac{1}{4}Q_1^2 & \frac{1}{2}Q_3 & & &		-\frac{1}{2}Q_1 &I_n&& \\[5pt]
		\frac{1}{2}Q_3 & Q_4 & \ddots & &				 & 0 &\ddots& \\[5pt]
		&\ddots & \ddots & \frac{1}{2}Q_{2m-1}	&		&  & \ddots & I_n \\[5pt]
		&& \frac{1}{2}Q_{2m-1} & Q_{2m}&				&& & 0
	\end{array}\right] .
\end{align*}
By \Cref{Mr_lem}, $x I - M_r$ is a linearization of $\rev Q(x)$.
 Since $Q(x)\ge0$ for all $x\in\RR$, $\rev Q(x) = x^{2m} Q\left(\frac{1}{x}\right) \ge0$ for all nonzero $x\in\RR$.  Thus by continuity, $\rev Q(x) \ge0$ for all $x\in\RR$.  Then by \Cref{multiplicities}, 
the partial multiplicities of every real eigenvalue of $\rev Q(x)$ are all even.
Since $x I - M_r$ is a linearization of $\rev Q(x)$, all the partial multiplicities of every eigenvalue of $M_r$ and $\rev Q(x)$ are the same, so the partial multiplicities of every real eigenvalue of $M_r$ are all even. 
Now since all roots of $\det(Q(x))$ have even multiplicity, all roots of $\det(\rev Q(x))$ have even multiplicity (note the multiplicity of zero as a root is also known to be even since $Q(x)$ has even degree).  Then since
$$\det(x I-M_r) = \det(\rev Q(x)),$$
all eigenvalues of $M_r$ have even algebraic multiplicity.  In particular, all non-real eigenvalues of $M_r$ have even algebraic multiplicity.  Hence by \Cref{L812} and \Cref{T8110}, there is a skew-symmetric solution, $\tilde{X}$ of \cref{QARE}.
Then since the Schur complement with respect to $I_n$ is zero, we know
$$F_{\tilde{X}}
	= \begin{bmatrix}
		I_n & \Gamma_{12}+B^T\tilde{X} \\
		\Gamma_{12}^T-\tilde{X}B & \Gamma_{22}+A^T\tilde{X} -\tilde{X}A
	\end{bmatrix} \ge0 ,$$
and
\begin{align*}
\rank \left(F_{\tilde{X}}\right)
	& =\rank\left(\begin{bmatrix}
				I_n & \Gamma_{12}+B^T\tilde{X} \\
				\Gamma_{12}^T-\tilde{X}B & \Gamma_{22}+A^T\tilde{X} -\tilde{X}A
			\end{bmatrix}\right)  
	= \rank I_n = n.
\end{align*}
We can factorize
$$F_{\tilde{X}}=\begin{bmatrix}
	I_n & \Gamma_{12}+B^T\tilde{X} \\
	\Gamma_{12}^T-\tilde{X}B & \Gamma_{22}+A^T\tilde{X} -\tilde{X}A
	\end{bmatrix}
	=\begin{bmatrix} G_0^T \\ \vdots \\ G_m^T \end{bmatrix}
	\begin{bmatrix} G_0 & \cdots & G_m \end{bmatrix} , $$
with $G_0=I_n$ and $G_i$ real $n\times n$ matrices for $i = 1, 2, \dots,m$.  Then,
\begin{align*}
Q(x) &= \begin{bmatrix} I_n & xI_n &\cdots &x^mI_n \end{bmatrix}
		F_{\tilde{X}}
		\begin{bmatrix} I_n \\ xI_n \\ \vdots  \\ x^mI_n \end{bmatrix} \\[5pt]
	&= \begin{bmatrix} I_n & xI_n &\cdots &x^mI_n \end{bmatrix}
		\begin{bmatrix} G_0^T \\ \vdots \\ G_m^T \end{bmatrix}
	\begin{bmatrix} G_0 & \cdots & G_m \end{bmatrix}
		\begin{bmatrix} I_n \\ xI_n \\ \vdots  \\ x^mI_n \end{bmatrix}.
\end{align*}
Thus for
$G(x)= \sum_{j=0}^{m}G_jx^j,$
$Q(x) = G(x)^TG(x).$
\end{proof}

\noindent \Cref{sq_factor} required $Q_0>0$.  We can relax this condition as follows.

\begin{cor}\label{sq_factor_cor}
Let $Q(x)= \sum_{j=0}^{2m}Q_jx^j$ be a real $n\times n$ regular symmetric positive semidefinite matrix polynomial of degree $2m$.  Then all roots of $\det(Q(x))$ have even multiplicity if and only if there exists an $n\times n$ real matrix polynomial $G(x)= \sum_{j=0}^{m}G_jx^j$
of degree $m$ such that
$$Q(x) = G(x)^TG(x).$$
\end{cor}
\begin{proof}
First, assume $Q(x) = G(x)^TG(x).$ Then $\det(Q(x)) = \det(G(x))^2$, so clearly all roots have even multiplicity.  On the other hand, assume all roots of $\det(Q(x))$ have even multiplicity. Let $x_0\in\RR$ be such that $\det(Q(x_0))\ne 0$ (note such $x_0$ exists since $Q(x)$ is regular, i.e. the determinant is not identically zero).  Consider 
$$P(x):= Q(x_0-x).$$
Then $P(x)$ is  an $n\times n$ real symmetric matrix polynomial of degree $2m$ such that 
$$P(0) = Q(x_0)>0 \qquad \text{and} \qquad \det(P(x)) = \det(Q(x_0-x)) .$$
Thus the roots of $\det(P(x))$ all have even multiplicity, so by \Cref{sq_factor}, there is an $n\times n$ real matrix polynomial $H(x)= \sum_{j=0}^{m}H_jx^j$
of degree $m$ such that
$$P(x) = H(x)^TH(x).$$
Define
$$G(x) : = H(x_0-x).$$
Then $G(x)= \sum_{j=0}^{m}G_jx^j$ is an $n\times n$ real matrix polynomial such that
\begin{align*}
Q(x) = P(x_0-x) = H(x_0-x)^TH(x_0-x)=G(x)^TG(x).
\end{align*}
\end{proof}

\section{$M_r$-invariant $\hat{H}_r$-neutral subspace}
The proof of \Cref{sq_factor} is constructive.  It hinges on finding the real skew-symmetric solution $X$ to the modified algebraic Riccati equation.  Back in \Cref{ricatti}, we found such a solution by constructing an $mn$-dimensional $M_r$-invariant $\hat{H}_r$-neutral subspace.  Following the proof of \Cref{A_invariant}, found in \cite[Theorem 2.6.3]{Ric}, also to be found in Gohberg, Lancaster, and Rodman's later book \textit{Indefinite Linear Algebra and Applications} \cite[Theorem I.3.21]{Indefinite}, to find such a subspace, we must first convert $M_r$ to its real Jordan form.  For this, assume there are $k$ Jordan blocks corresponding to real eigenvalues $\lambda_1,\lambda_2,\dots,\lambda_k$ (note some eigenvalues may be repeated as they can occur in multiple blocks).  The block corresponding to eigenvalue $\lambda_j$ has size $r_j$ and is denoted by
$$J_{r_j}(\lambda_j) = \begin{bmatrix}
					\lambda_j & 1 & & \\
					& \ddots & \ddots & \\
					&&\lambda_j & 1 \\
					&&&\lambda_j  
					\end{bmatrix} \in \RR^{r_j\times r_j}.$$
Next, assume there are $\ell$ real Jordan blocks corresponding to pairs of non-real eigenvalues $\alpha_1\pm i\beta_1, \alpha_2\pm i\beta_2, \dots, \alpha_{\ell}\pm i \beta_{\ell}$ (note again some eigenvalues may be repeated as they can occur in multiple blocks).
The block corresponding to the pair of eigenvalues $\alpha_j\pm i \beta_j$ has size $2s_{j}$ and is denoted by
$$J_{2s_{j}}(\alpha_j\pm i \beta_j) = \begin{bmatrix}
					C(\alpha_j,\beta_j) & I_2 & & \\
					& \ddots & \ddots & \\
					&&C(\alpha_j,\beta_j)  & I_2 \\
					&&&C(\alpha_j,\beta_j)
					\end{bmatrix} \in \RR^{2s_{j}\times 2s_{j}},
	\quad \text{where} \quad
	C(\alpha_j,\beta_j) = \begin{bmatrix}
							\alpha_j & \beta_j \\
							-\beta_j & \alpha_j
							\end{bmatrix}.$$
Then for some real invertible matrix $S$,
$$M_r = SJS^{-1},$$
where
\begin{equation}\label{Jordan}
J = J_{r_1}(\lambda_1)
	\oplus \cdots\oplus J_{r_k}(\lambda_{k})\oplus J_{2s_{1}}(\alpha_1\pm i\beta_1) 
	\oplus \cdots \oplus J_{2s_{\ell}}(\alpha_{\ell}\pm i \beta_{\ell}).
\end{equation}
We will form the desired $mn$-dimensional $M_r$-invariant $\hat{H}_r$-neutral subspace by extracting $nm$ columns of $S$ and putting these columns together to form a $2nm\times nm$ matrix $Y$. The desired $M_r$-invariant $\hat{H}_r$-neutral subspace is then the column space of this matrix $Y$. The construction of $Y$ is outlined as follows.

\vspace{5pt}
\noindent{\bf Construction of $Y$:}
		\begin{enumerate}[label=\arabic*.]
			\item For each Jordan block $J_{r_j}(\lambda_j)$ of a real eigenvalue, there are $r_j$  corresponding columns in $S$.  Note $r_j$ is known to be even.  Take the first $\frac{r_j}{2}$ of those columns.
			\item For each real Jordan block $J_{2s_j}(\alpha_j\pm i\beta_j)$ of a complex conjugate pair of eigenvalues, there are $2s_j$ corresponding columns in $S$.  If $s_j$ is even, take the first $s_j$ of those columns.
			\item Each remaining real Jordan block $J_{2s_j}(\alpha_j\pm i\beta_j)$ of a complex conjugate pair of eigenvalues has $2s_j$ corresponding columns in $S$, where $s_j$ is odd.  Since the algebraic multiplicity of each eigenvalue is even, we can pair up each of these blocks with another such block of the same eigenvalue, say $J_{2s_j}(\alpha_j\pm i\beta_j)$ pairs with $J_{2s_p}(\alpha_p\pm i\beta_p)$ where $\alpha_j=\alpha_p$ and $\beta_j=\beta_p$.  Take the first $s_j-1$ of the columns of $S$ corresponding to $J_{2s_j}(\alpha_j\pm i\beta_j)$ and the first $s_p-1$ of the columns of $S$ corresponding to $J_{2s_p}(\alpha_p\pm i\beta_p)$.  Lastly, take
			\begin{enumerate}
	\item The $s_j$th column of $S$ corresponding to $J_{2s_j}(\alpha_j\pm i\beta_j)$ plus the $s_p+1$st column of $S$ corresponding to $J_{2s_p}(\alpha_p\pm i\beta_p)$.
	\item The $s_j+1$st column of $S$ corresponding to $J_{2s_j}(\alpha_j\pm i\beta_j)$ minus the $s_p$th column of $S$ corresponding to $J_{2s_p}(\alpha_p\pm i\beta_p)$.
\end{enumerate}
		\end{enumerate}
Putting these columns together to form a $2nm\times nm$ matrix $Y$, we get that the desired $M_r$-invariant $\hat{H}_r$-neutral subspace as the column space of this matrix $Y$.  The reason why this works hinges on the following theorem.
\begin{thm}\cite[Theorem 2.6.1]{Ric} \label{T261}
Let $H$ be a nonsingular real $n\times n$ symmetric matrix and let $A$ be a real $H$-symmetric matrix.  Then there exists an invertible real matrix $S$ such that $J:=S^{-1}AS$ can be written as in \cref{Jordan} and $P:=S^*HS$ has the form
$$P = \epsilon_1P_{r_1}\oplus\cdots\oplus \epsilon_kP_{r_k}\oplus P_{2s_1} \oplus\cdots\oplus P_{2s_\ell},$$
where $\lambda_1,\dots,\lambda_k,$ $\alpha_1,\dots,\alpha_\ell$, and $\beta_1,\dots,\beta_\ell$ are positive numbers; $\epsilon_j = \pm 1, j=1,2,\dots,k$; $P_j$ is the $j\times j$ reversal matrix, i.e.
$$P_j = \begin{bmatrix} 
			&&1\\
			&\iddots&\\
			1&&
			\end{bmatrix}
			\in\RR^{j\times j}.$$
Moreover, the canonical form $(J,P)$ of $(A,H)$ is uniquely determined by $(A,H)$ up to permutation of blocks $(J_{r_j}(\lambda_j),\epsilon_jP_{r_j})$ for $j=1,\dots,k$ and $(J_{2s_j}(\alpha_j\pm i \beta_j),P_{2s_j})$ for $j=1,\dots,\ell$.
\end{thm}
\noindent
This theorem tells us that when we write
$$M_r = SJS^{-1},$$
we also have
$$\hat{H}_r = S^{-*}PS^{-1},$$
where $P$ is as defined in the theorem and $S^{-*}=(S^*)^{-1}.$  We illustrate three cases with simple examples.  
First assume $J = J_{r_1}(\lambda_1)$, where we know $r_1$ is even.  Then by \Cref{T261}, $P=\epsilon_1P_{r_1}$ where $\epsilon = \pm1$. Let $Y$ be the matrix formed from the first $\frac{r_1}{2}$ columns of $S$.  We claim that the column space of $Y$ is $M_r$-invariant and $\hat{H}_r$-neutral.  Indeed,
$$M_rY= SJS^{-1}Y
			=SJ
				\begin{bmatrix}
					I_{r_1/2} \\
					0_{r_1/2}
				\end{bmatrix}
				=YJ_{r_1/2}(\lambda_1),$$
and 
$$ Y^* \hat{H}_r Y
	=  Y^* S^{-*}\epsilon_1P_{r_1}S^{-1} Y
	= \epsilon_1\begin{bmatrix}
					I_{r_1/2} &
					0_{r_1/2}
				\end{bmatrix}
				\begin{bmatrix}
					0_{r_1/2} & P_{r_1/2}\\
					P_{r_1/2} & 0_{r_1/2}
				\end{bmatrix}
				\begin{bmatrix}
					I_{r_1/2} \\
					0_{r_1/2}
				\end{bmatrix}
				=0.$$
Next assume $J = J_{2s_1}(\alpha_1\pm i\beta_1)$, where $s_1$ is even.  Then by \Cref{T261}, $P=P_{2s_1}$. Let $Y$ be the matrix formed from the first $s_1$ columns of $S$.  Then the column space of $Y$ is $M_r$-invariant and $\hat{H}_r$-neutral.  Indeed,
$$M_rY= SJS^{-1}Y
				=YJ_{s_1}(\alpha_1\pm i \beta_1),$$
and 
$$ Y^* \hat{H}_r Y
	=  Y^* S^{-*}P_{2s_1}S^{-1} Y
	= \begin{bmatrix}
					I_{s_1} &
					0_{s_1}
				\end{bmatrix}
				\begin{bmatrix}
					0_{s_1} & P_{s_1}\\
					P_{s_1} & 0_{s_1}
				\end{bmatrix}
				\begin{bmatrix}
					I_{s_1} \\
					0_{s_1}
				\end{bmatrix}
				=0.$$
Finally, assume 
$$J = J_{2s_1}(\alpha_1\pm i\beta_1) \oplus J_{2s_2}(\alpha_2\pm i\beta_2)$$
where $s_1$ and $s_2$ are odd, $\alpha_1=\alpha_2$, and $\beta_1=\beta_2$. Form the $(s_1+s_2)\times(s_1+s_2)$ matrix $Y$ using columns $1,2,\dots,s_1-1$ and $2s_1+1,2s_1+2, 2s_1+s_2-1$ from $S$.  Additionally, form the second to last column of $Y$ as column $s_1$ of $S$ plus column $2s_1+s_2+1$ of $S$.  Finally, form the last column of $Y$ as column $s_1+1$ of $S$ minus column $2s_1+s_2$ of $S$.  
Then the column space of $Y$ is $M_r$-invariant and $\hat{H}_r$-neutral.  Indeed,  let $S_j$ denote column $j$ of $S$.  Let $\alpha=\alpha_1=\alpha_2$ and $\beta=\beta_1=\beta_2$.  Then
$$Y = \begin{bmatrix} S_1 & \cdots & S_{s_1-1} & S_{2s_1+1} & \cdots &S_{2s_1+s_2-1} & S_{s_1}+S_{2s_1+s_2+1} & S_{s_1+1}-S_{2s_1+s_2} \end{bmatrix}.$$
Letting $0_j$ denote the $j\times j$ zero matrix (and generic 0 may represent a zero matrix whose size can be inferred in the context of the other elements of the matrix), we have
$$M_rY = SJS^{-1}Y = SJ \begin{bmatrix}
			I_{s_1-1} & 0 & 0 & 0 \\
			0 & 0 &1 & 0\\
			0 & 0 &0 &1\\
			0_{s_1-1} & 0 & 0 & 0\\
			0 & I_{s_2-1} & 0 & 0 \\
			0 & 0 &0 & -1\\
			0 & 0 &1 &0\\
			0 & 0_{s_2-1} & 0 & 0		
			\end{bmatrix}
			=S
			\begin{bmatrix}
				J_{s_1-1}(\alpha\pm i\beta)  & 0 & * & * \\
				0 & 0 &\alpha & \beta\\
				0 & 0 &-\beta &\alpha\\
				0_{s_1-1} & 0 & 0 & 0\\
				0 & J_{s_2-1}(\alpha\pm i \beta)  & * & * \\
				0 & 0 &\beta & -\alpha\\
				0 & 0 &\alpha &\beta\\
				0 & 0_{s_2-1} & 0 & 0
			\end{bmatrix},$$
 so $Y$ is $M_r$-invariant. Next, by \Cref{T261}, we know $\hat{H}_r = S^{-*}PS$ where 
$$P = P_{2s_1}\oplus P_{2s_2}.$$
Then
\begin{align*}
 Y^* \hat{H}_r Y
	& =  Y^* S^{-*}(P_{2s_1}\oplus P_{2s_2})S^{-1} Y \\[5pt]
	& =\begin{bmatrix}
			I_{s_1-1} & 0 &0&0_{s_1-1}&0&0&0&0\\
			0&0&0&0&I_{s_2-1}&0&0&0_{s_2-1} \\
			0&1&0&0&0&0&1&0\\
			0&0&1&0&0&-1&0&0
			\end{bmatrix}
			(P_{2s_1}\oplus P_{2s_2})
			\begin{bmatrix}
			I_{s_1-1} & 0 & 0 & 0 \\
			0 & 0 &1 & 0\\
			0 & 0 &0 &1\\
			0_{s_1-1} & 0 & 0 & 0\\
			0 & I_{s_2-1} & 0 & 0 \\
			0 & 0 &0 & -1\\
			0 & 0 &1 &0\\
			0 & 0_{s_2-1} & 0 & 0
			\end{bmatrix}. \\[5pt]
	& =\begin{bmatrix}
			I_{s_1-1} & 0 &0&0_{s_1-1}&0&0&0&0\\
			0&0&0&0&I_{s_2-1}&0&0&0_{s_2-1} \\
			0&1&0&0&0&0&1&0\\
			0&0&1&0&0&-1&0&0
			\end{bmatrix}
			\begin{bmatrix}
			0_{s_1-1} & 0 & 0 & 0 \\
			0 & 0 &0 & 1\\
			0 & 0 &1 &0\\
			P_{s_1-1} & 0 & 0 & 0\\
			0 & 0_{s_2-1} & 0 & 0 \\
			0 & 0 &1 & 0\\
			0 & 0 &0 &-1\\
			0 & P_{s_2-1} & 0 & 0
			\end{bmatrix}=0.
\end{align*}
These three cases together show why the general procedure for picking columns of $S$ generate the desired subspace.

\section{Algorithm}
Now that we have the construction of the invariant subspace, we can put everything together to get an explicit algorithm as outlined below.  To the best of our knowledge, no algorithm for finding the real factorization $Q(x) = G(x)^TG(x)$ exists in the literature. \\[10pt]
%
\hrule height 1pt 
\vspace{3pt}
\noindent
\textbf{Algorithm 1} Real Factorization of PSD Matrix Polynomial
\vspace{3pt}
\hrule
\noindent
\textbf{Input:} A real $n\times n$ regular symmetric positive semidefinite matrix polynomial $Q(x) = \sum_{j=0}^{2m} Q_jx^j$ for which $\det(Q(x))$ has only roots of even multiplicity.
\begin{enumerate}[label=\arabic*.]
	\item Fix $x_0\in\RR$ such that $\det(Q(x_0))\ne 0$.
	\item Set $\hat{P}(x) = Q(x_0-x) =: \sum_{j=0}^{2m} \hat{P}_jx^j$.
	\item Set $P(x) = \hat{P}_0^{-1/2}\hat{P}(x)\hat{P}_0^{-1/2} =: \sum_{j=0}^{2m} P_jx^j.$
	\item Set 
	$$M_r = \left[\begin{array}{cccc|ccccc}
		-\frac{1}{2}P_1 &&&&		-I_n &&& \\[5pt]
		I_n & 0 &&&			 & 0 && \\
		& \ddots & \ddots &	&	& & \ddots & \\
		&& I_n & 0	&		&& & 0 \\[5pt]\hline
		&&&&&&&&\\
		P_2-\frac{1}{4}P_1^2 & \frac{1}{2}P_3 & & &		-\frac{1}{2}P_1 &I_n&& \\[5pt]
		\frac{1}{2}P_3 & P_4 & \ddots & &				 & 0 &\ddots& \\[5pt]
		&\ddots & \ddots & \frac{1}{2}P_{2m-1}	&		&  & \ddots & I_n \\[5pt]
		&& \frac{1}{2}P_{2m-1} & P_{2m}&				&& & 0
	\end{array}\right] .$$
	\item Find the real Jordan canonical form $M_r = SJS^{-1}$.
	\item Choose columns of the matrix $S$ to form a $2nm\times nm$ matrix $Y$ as outlined under `Construction of $Y$' in the previous section.

	\item Define the $nm\times nm$ matrices $X_1$ and $X_2$ by 
		$Y = \begin{bmatrix} X_1 \\X_2\end{bmatrix}$
		and compute $X = X_2X_1^{-1}$.
	\item Set
	$$F_X = 
	\begin{bmatrix}
			I_n & \frac{1}{2}P_1 & & \\[10pt]
			\frac{1}{2}P_1 & P_2 & \ddots & \\[10pt]
			&\ddots & \ddots & \frac{1}{2}P_{2m-1} \\[10pt]
			&& \frac{1}{2}P_{2m-1} & P_{2m}
		\end{bmatrix} 
			+ \begin{bmatrix} 0_{nm\times n} & X \\ 0_n & 0_{n\times nm} \end{bmatrix} 
			- \begin{bmatrix} 0_{n\times nm} & 0_n \\ X & 0_{nm\times n}\end{bmatrix}
		.$$
	\item By construction, $F_X$ is  positive semidefinite of rank $n$, so factorize as
		$$F_X=\begin{bmatrix} H_0^T \\ \vdots \\ H_m^T \end{bmatrix}
		\begin{bmatrix} H_0 & \cdots & H_m \end{bmatrix}, $$
		where $H_j$ is real $n\times n$ for all $j$.
	\item Set $H(x) = \sum_{j=0}^m H_jx^j$.
	\item Set $G(x) = \sum_{j=0}^m P_0^{1/2}H_j(x_0-x)^j =: \sum_{j=0}^mG_jx^j.$
\end{enumerate}
\textbf{Output:} A real $n\times n$ matrix polynomial $G(x) = \sum_{j=0}^{m} G_jx^j$ such that
$$Q(x) = G(x)^TG(x).$$
\hrule height 1pt
\vspace{3pt}

%
\subsection{Examples}
Let us illustrate the above ideas in a few examples.  Note that the Jordan canonical factorization in the algorithm relies on exact computation, so the examples were computed in Maple. The first example is the real eigenvalue case.  The second example is the non-real eigenvalue case with even $s_j$.  

\begin{example}
\normalfont
Take 
$$Q(x) = \begin{bmatrix}
		2x^2+2x+1 & -4x^2-3x \\
		-4x^2-3x &  8x^2+4x+1
		\end{bmatrix}
		=\begin{bmatrix} 
			1&0\\0&1
			\end{bmatrix}
			+x\begin{bmatrix}
				2&-3\\-3&4 
				\end{bmatrix}
			+x^2\begin{bmatrix}
				2&-4\\-4&8
				\end{bmatrix}.$$
Then $M_r = SJS^{-1} $ for
$$J = \begin{bmatrix}
    0 &    1   &  0  &   0\\
     0   & 0   &  0  &   0\\
     0    & 0   &  -3   &  1\\
     0    & 0   &  0   &  -3
	\end{bmatrix}
, \ 
S = \begin{bmatrix}
   \frac{2}{9} & \frac{37}{54} & \frac{-5}{18} & \frac{17}{54} \\[5pt]
   \frac{1}{9} & \frac{10}{27} & \frac{5}{18} & \frac{-10}{27} \\[5pt]
   \frac{-1}{18} & \frac{-19}{54} & \frac{-5}{36} & \frac{19}{54} \\[5pt]
   \frac{1}{9} & \frac{19}{108} & \frac{-5}{36} & \frac{-19}{108}
   \end{bmatrix}.
$$
We have $J = J_2\left( 0\right) \oplus J_2\left( -3\right)$, so $r_1 = r_2 = 2$.  Thus we take the 1st column of $S$ corresponding to the first Jordan block as well as the 1st column of $S$ corresponding to the  second Jordan block.
$$\begin{bmatrix}
	  \frac{2}{9} &   \frac{-5}{18}\\[5pt]
      \frac{1}{9}& \frac{5}{18}  \\[5pt]
   \frac{-1}{18} & \frac{-5}{36} \\[5pt]
  \frac{1}{9} &   \frac{-5}{36}
   \end{bmatrix} =: \begin{bmatrix} X_1 \\ X_2 \end{bmatrix}. $$
 Our invariant subspace is thus 
 $$\Image \begin{bmatrix} X_1 \\ X_2 \end{bmatrix} = \Image \begin{bmatrix} I \\ X_2 X_1^{-1}\end{bmatrix}.$$
Then,
 $$X = X_2X_1^{-1} = \begin{bmatrix} 0 & \frac{-1}{2} \\[3pt] \frac{1}{2} & 0 \end{bmatrix}.$$
Thus 
$$F_{{X}}=F_0 + \begin{bmatrix} 0 & {X} \\ 0 & 0\end{bmatrix} - \begin{bmatrix} 0 & 0 \\ {X} & 0\end{bmatrix} 
	=\begin{bmatrix}
     1 &    0  &   1  &  -2\\
     0 &    1  &  -1  &   2\\
     1 &   -1  &   2  &  -4\\
    -2 &    2  &  -4  &   8
		\end{bmatrix}.$$
We factorize as 
$$F = \begin{bmatrix}
		1&0\\
		0&1\\
		1&-1\\
		-2&2
		\end{bmatrix}
		\begin{bmatrix}
		1&0&1&-2\\
		0&1&-1&2
		\end{bmatrix}.$$
Thus 
$$G_0 = \begin{bmatrix}
			1&0\\
			0&1
			\end{bmatrix},
\qquad
G_1 = \begin{bmatrix}
		1&-2\\
		-1&2
	\end{bmatrix},
\qquad
G(x) = G_0+xG_1.$$
We can verify $Q(x) = G(x)^TG(x).$
\end{example}

\begin{example}
\normalfont
Take now
$$Q(x) = \begin{bmatrix}
			2x^2+2x+1 & x^2+2x \\
			x^2+2x & 13x^2+4x+1
		\end{bmatrix}
		=\begin{bmatrix} 
			1&0\\0&1
			\end{bmatrix}
			+x\begin{bmatrix}
				2&2\\2&4 
				\end{bmatrix}
			+x^2\begin{bmatrix}
				2&1\\1&13
				\end{bmatrix}.$$
Then $M_r = SJS^{-1} $ for
$$J = \begin{bmatrix}
	\frac{-3}{2} & \frac{\sqrt{11}}{2} & 1 & 0 \\[5pt]
	\frac{-\sqrt{11}}{2} & \frac{-3}{2} & 0 & 1 \\[5pt]
	0 & 0 & \frac{-3}{2} & \frac{\sqrt{11} }{2}\\[5pt]
	0 & 0 & \frac{-\sqrt{11}}{2} & \frac{-3}{2}
	\end{bmatrix}
, \ 
S = \begin{bmatrix}
    \frac{4}{11} & \frac{-2\sqrt{11}}{11} & \frac{1}{2} & \frac{5\sqrt{11}}{242} \\[5pt]
    \frac{-3}{11} & \frac{-\sqrt{11}}{11} & 0 & \frac{5\sqrt{11}}{121} \\[5pt]
    \frac{-6}{11} & \frac{-2\sqrt{11}}{11} & 0 & \frac{-12\sqrt{11}}{121} \\[5pt]
    \frac{-8}{11} & \frac{4\sqrt{11}}{11} & 0 & \frac{6\sqrt{11}}{121}
   \end{bmatrix}.
$$
We have $J = J_4\left(\frac{3}{2}\pm i \frac{\sqrt{11}}{2}\right)$, so $s_1 = 2$.  Thus we take the first two columns of $S$ corresponding to the first and only Jordan block.
$$\begin{bmatrix}
	 \frac{4}{11} & \frac{-2\sqrt{11}}{11}  \\[5pt]
    \frac{-3}{11} & \frac{-\sqrt{11}}{11} \\[5pt]
    \frac{-6}{11} & \frac{-2\sqrt{11}}{11} \\[5pt]
    \frac{-8}{11} & \frac{4\sqrt{11}}{11}
   \end{bmatrix} =: \begin{bmatrix} X_1 \\ X_2 \end{bmatrix} .$$
 Our invariant subspace is thus 
 $$\Image \begin{bmatrix} X_1 \\ X_2 \end{bmatrix} = \Image \begin{bmatrix} I \\ X_2 X_1^{-1}\end{bmatrix}.$$
Then,
 $$X = X_2X_1^{-1} = \begin{bmatrix} 0 & 2 \\ -2 & 0 \end{bmatrix}.$$
Thus 
$$F_{{X}}=F_0 + \begin{bmatrix} 0 & {X} \\ 0 & 0\end{bmatrix} - \begin{bmatrix} 0 & 0 \\ {X} & 0\end{bmatrix} 
	=\begin{bmatrix}
		1&0&1&3\\
		0&1&-1&2\\
		1&-1&2&1\\
		3&2&1&13
		\end{bmatrix}.$$
We factorize as 
$$F = \begin{bmatrix}
		1&0\\
		0&1\\
		1&-1\\
		3&2
		\end{bmatrix}
		\begin{bmatrix}
		1&0&1&3\\
		0&1&-1&2
		\end{bmatrix}.$$
Thus 
$$G_0 = \begin{bmatrix}
			1&0\\
			0&1
			\end{bmatrix},
\qquad
G_1 = \begin{bmatrix}
		1&3\\
		-1&2
	\end{bmatrix},
\qquad
G(x) = G_0+xG_1.$$
We can verify $Q(x) = G(x)^TG(x).$
\end{example}

\subsection{Numerical Considerations}

Note that the algorithm presented here relies on being able to compute the real Jordan canonical form of a matrix.  This is not numerically stable and thus the algorithm assumes exact computation.  However, if all  eigenvalues have algebraic multiplicity 2 and geometric multiplicity 1 (i.e. in \cref{Jordan}, $r_j=2, j=1,2,\dots,k$ and $s_j=2, j=1,2,\dots,\ell$), we can still find the required $M_r$-invariant $\hat{H}_r$-neutral subspace.  Indeed, for size 2 blocks of real eigenvalues, our algorithm says to take the first column of $S$ associated with that block.  This is simply the eigenvector associated with that eigenvalue.  Similarly, for a size 4 block of a pair of non-real eigenvalues, our algorithm says to take the first two columns of $S$ associated with that block.  These are the real and imaginary parts of the eigenvector associated with that eigenvalue pair.   In short, all we need here are the eigenvectors of the $M_r$ matrix.  

We tested a numerical implementation of our algorithm in MATLAB by running 100 trials.  In each trial, a random matrix size $n$ and degree $m$ were chosen between 2 and 8.  A matrix polynomial $G(x)=\sum_{i=0}^m G_ix^i$ of size $n\times n$ and degree $m$ was created by taking $G_0=I_n$ and randomly generating $G_1, \dots G_m$.  Then the coefficients $Q_0,Q_1,\dots,Q_{2m}$ of the matrix polynomial $Q(x) = \sum_{i=0}^{2m} Q_ix^i$  were computed by setting $Q(x) = G(x)^TG(x)$.  Operating under the assumption that all  eigenvalues have algebraic multiplicity 2 and geometric multiplicity 1, Algorithm 1 was implemented, where the matrix $Y$ in step 6 was formed from the eigenvectors of $M_r$ as outlined in the previous paragraph.  The algorithm outputted coefficients $\hat{G}_i$ for the matrix polynomial $\hat{G}(x) = \sum_{i=0}^m \hat{G}_ix^i$.  We computed coefficients $\hat{Q}_i$ for the matrix polynomial $\hat{Q}(x) = \hat{G}(x)^T\hat{G}(x)$.  Since in application, only $Q(x)$ and not $G(x)$ would be known before the implementation of the algorithm, the error of the trial was calculated as the maximum absolute entry of 
$$\begin{bmatrix} Q_0 & Q_1 & \cdots & Q_{2m} \end{bmatrix}
	-\begin{bmatrix} \hat{Q}_0 & \hat{Q}_1 & \cdots & \hat{Q}_{2m} \end{bmatrix}.$$
The overall worst error among all 100 trials was on the order of $10^{-6}$.  It is important to note, though, that for particular examples where the multiplicities of eigenvalues are higher, the algorithm does not work in any reliable way.  

Finally, one can also use numerical methods to conclude reliably that if an eigenvalue of $M_r$ has algebraic multiplicity equal to one, and then the corresponding matrix polynomial $Q(x)$ will not have a real factorization.

\section*{Acknowledgement}
We thank the referee for their careful reading of the manuscript and their detailed advice, which, among other things, lead to the addition of the subsection on numerical considerations.

%
\bibliographystyle{plain}
\bibliography{FPaperV4}

\end{document}